\documentclass[11pt]{article}
\usepackage[stretch=10]{microtype}
\usepackage[a4paper, total={6in, 8in}]{geometry}
\usepackage{authblk}
\usepackage{bbold}
\usepackage{mathtools,amsthm}
\usepackage{dutchcal}
\usepackage[hidelinks,colorlinks=true,linkcolor=blue,citecolor=blue]{hyperref}
\usepackage{xcolor}
\usepackage[shortlabels]{enumitem}
\usepackage{tikz}
\usepackage{nicematrix}
\usepackage[
  backend=biber,
  style=alphabetic,
  maxbibnames=4,
  sorting=nyt
]{biblatex}
\newcommand{\ZZ}{\mathsf Z}

\newcommand{\dts}{.\,.\,}
\newtheorem{theorem}{Theorem}
\theoremstyle{lemma}
\newtheorem{lemma}{Lemma}
\theoremstyle{remark}
\newtheorem{remark}{Remark}

\title{Particle detection among Random Walks as a non-reversible Random Interlacements Process}
\author[1]{Carlos Martinez}
\author[2]{Gonzalo Panizo}
\affil[1]{Instituto de Matemática y Ciencias Afines, Universidad Nacional de Ingeniería, Lima, Perú.}
\affil[2]{Instituto de Matemática y Ciencias Afines, Universidad Nacional de Ingeniería, Lima, Perú.\\E-mail: gonzalo@imca.edu.pe}

\begin{document}
\maketitle

\begin{abstract}
  Consider the random set composed of particles initially distributed on \( \mathsf{Z}^d \),
  \( d \geq 2 \), according to a Poisson point process of intensity \( u > 0 \) and moving as
  independent simple symmetric random walks, the trap particles. We are interested in the detection
  by these particles of a target particle, initially at the origin and able to move with finite
  mean speed. The escape strategy for the target particle is to stay inside the infinite
  cluster of empty sites, assuming \( u \) is in the subcritical site percolation regime of
  particle occupation.  By translating the problem to the framework of percolation of Random
  Interlacements we also prove that for \( u \) large enough the target doesn't escape. In doing
  this we extend the random interlacements formalism in order to allow non reversible random
  walks. As far a we know this is the first example of Random Interlacements for non-reversible
  Markov chains.
\end{abstract}

\section{Introduction} 
\label{sec:introduction}

\par The motivation for this article is twofold.  One is to treat a particular problem of escaping from
randomly moving traps, also known as target detection problem. The other is to show that the Random
Interlacements basic framework \cite{sznitman2010vacant} is flexible enough to allow non-reversible
path laws and make use of the renormalization machinery already developed for percolation problems
of Random Interlacements to address one aspect of the problem just mentioned.

\par We begin by formulating the target detection problem and stating our main result. See
\cite{sidoravicius2015phase}, \cite{stauffer2015space}, \cite{drewitz2011survival},
\cite{athreya2019random} and \Cite{peres2013mobile} for slightly different views of this
problem. In particular Sidoravicius and Stauffer \cite{sidoravicius2015phase} solve a very similar
problem where a target particle moving with bounded speed escapes from a uniformly distributed set
of trap particles moving as independent simple symmetric random walks on \( \mathsf{Z}^d \).  The
escape strategy they propose is with the target particle moving at all times away from the origin,
which seems somewhat artificial in view of the symmetry of the problem. The vision they adopt about
the problem is that of a directed percolation problem in space-time, which demands the target
particle to be clairvoyant in the sense that each step of the target particle depends on the whole
history of the trap particles. In contrast our escape strategy is adapted, and in fact previsible,
but the price we pay is that the speed of the target particle is bounded only in average.

\par First we need some notation. Convene from now on that the natural numbers start at \( 0 \), i.e.
\( \mathsf{N} = \{0,1, \dots \} \). The max and \( l_1 \) norms in \( \mathsf{Z}^d \) are defined
as \( |x|_\infty = \max \{x_1, \dots ,x_d\} \) and \( |x|_1=|x_1|+\cdots+|x_d| \), and denote by
\( B(x,r)=\{y \in \mathsf{Z}^d: |x-y|_\infty \le r\} \) the ball centered at \( x \in \ZZ^d \) with
radius \( r>0 \). Vertices \( x,y \in \mathsf{Z}^d\) with \( |x-y|_1=1 \) are said to be neighbors,
denoted as \( x \sim y \). We say a set \( K\subset \mathsf{Z}^d \) is connected, if for every pair
of vertices \( x,y \) in \( K \) there is a finite sequence of neighboring vertices in \( K \)
going from \( x \) to \( y \), that is \( x=x_0\sim x_1\sim \cdots \sim x_n=y \), \( x_i\in K \)
\( \forall i \). The inside and outside boundary of a subset \( K \subset \mathsf{Z}^d \),
\( \partial^iK \) and \( \partial^oK \), are, respectively, the set of vertices in \( K \) having
some neighbor not in \( K \), the set of vertices not in \( K \) having some neighbor in \( K \).
Finally, \( K^\circ=K\setminus\partial^iK \) denotes the interior and \(|K|\) the cardinality of
\(K\).

\par Consider now a set of point particles initially distributed as a Poisson point process
\( \Pi^u \) of intensity \( u >0 \) over \( \mathsf{Z}^d \). We refer to the points of \( \Pi^u \)
as traps, and let each trap move independently as a discrete time simple symmetric random walk on
\( \mathsf{Z}^d \). We want to show that there is a nontrivial phase transition in the intensity
parameter \( u \) for each average speed \( s>0 \), in the sense that there exists a value
\( u _c(s) \in\,]0,\infty [ \) such that with positive probability a particle initially at the
origin and moving with a speed not exceeding \( s \) on average, escapes forever the traps for
\( u <u _c(s) \). Otherwise, if \( u >u_c(s) \) there is zero probability that the particle at the
origin moving with mean speed not exceeding \( s \) escapes.

More precisely, let us denote by \( \eta_0^j \), \( j\geq 1 \), the points of the Poisson point process
\( \Pi^u \) of intensity \( u>0 \), which will define the initial positions of the traps. Let
\begin{equation*}
  \eta_t^j=\eta_0^j + \zeta_t^j, \quad t\geq 0, j\geq 1
\end{equation*}
be their trajectories, where \( (\zeta_t^j)_{t\geq 0} \), \( j\geq 1 \) are independent simple symmetric random
walks on \( \mathsf{Z}^d \) starting at the origin.  We also introduce the space-time configuration
of traps, also called the environment, as the random element
\( \xi \in \{0,1\}^{\mathsf{Z}^d\times \mathsf{N}} \) given by \( \xi (x,t)=1 \) in case
\( \eta^j_t=x \) for some \( j\geq 1 \) and \( \xi (x,t)=0 \) otherwise, and denote its law by
\( P^u \). This can be seen as a non-independent site percolation model on
\( \mathsf{Z}^d\times \mathsf{N}\), where a site \( (x,t) \) is open if no trap occupy \( x \) at
time \( t \) and closed otherwise. The escape problem we are interested in can be rephrased as
the following directed percolation problem.

\par We say that \( \rho = (x_k,t_k)_{k\geq 0}\), where \( x_k\in \mathsf{Z}^d \),
\( t_k\in \mathsf{N} \), \( k\geq 0 \), is an escape trajectory for the target particle if
\begin{enumerate}[(1)]
\item \( (x_0,t_0),(x_1,t_1), \dots \), is a nearest neighbor path in
  \( \mathsf{Z}^d\times \mathsf{N} \subset \mathsf{Z}^{d+1} \), that is, for every \( k\geq 0 \),
  \( |(x_{k+1},t_{k+1}) - (x_k,t_k)|_1=1 \),
\item \( (x_0,t_0)=(0,0) \), \( (t_k)_{k\geq 0} \) is non-decreasing and \( \lim_{k\to\infty}t_k=+\infty \).
\end{enumerate}
The speed at time \( t \geq 0\) of an escape trajectory \( \rho \) is defined as
\begin{equation*}
  S_t=S_t(\rho)= \max\{i:t_i=t\}-\min\{i:t_i=t\}.
\end{equation*}
and its detection time in the environment \( \xi \) as
\begin{equation*}
  T_{\rm det}=T_{\rm det}(\rho)=t_\kappa, \text{ where } \kappa =\inf\{k\geq 0:\xi (x_k,t_k)=1\}.
\end{equation*}
A function associating to each realization of the environment \( \xi \) an escape trajectory
\( \rho(\xi) =(x_k,t_k)_{k\geq 0} \) is called an escape strategy. We say \( \rho(\xi) \)
is adapted if \( \rho_k(\xi)=(x_k,t_k) \in \mathcal{G}_{t_k}\) for all \( k\geq 0 \), where
\begin{equation*}
  \mathcal{G}_t=\sigma(\xi(x,s): x\in\mathsf{Z}^d, s\leq t)
\end{equation*}
is the time filtration generated by the environment.  Notice that for \( \rho \) escape strategy,
\( S_t=S_t(\rho) \) and \( T_{\rm det}=T_{\rm det}(\rho) \) are random variables depending on the
environment \( \xi \).  With this in mind we define the critical escape intensity for speed
\( s\geq 0 \) as
\begin{multline}
  \label{eq:lambda-critical}
  u_c(s)= \sup \{u \geq 0: \text{ exists \( \rho  \) adapted escape strategy}\\
  \text{with }
  P^u(T_{\rm det}=\infty)>0 \text{ and } E^u [\,S_t \,|\, T_d=\infty] \leq s \text{ for all }t\geq 0 \}.
\end{multline}
\begin{theorem}
  \label{thm:perc}
  Let \( d\geq 2 \). For every \( s>0 \), \( 0<u_c(s)<\infty \).
\end{theorem}
\begin{remark}
  Notice that the allowed percolation path is time-adapted, which is not a usual condition in
  percolation theory. In the proof of Theorem \ref{thm:perc} we show in fact that the escape
  strategy can be chosen to be previsible, that is \( \rho_k\in \mathcal{F}_{t_k-1} \) for all
  \( k \). It can even be chosen to satisfy \( \rho_k\in \mathcal{F}_{t_k-m} \), with
  \( m \in \mathsf{N} \) as big as desired. In exchange we need to have an infinite range vision at
  each time as the escape strategy constructed in the proof of Theorem \Ref{thm:perc} relies on
  knowing the position of the infinite cluster formed by the vacant set left by the traps at each
  time.
\end{remark}
\begin{lemma}
  \label{lm:percolation-component}
  Let \( \mathcal{C} \) be the infinite open cluster in supercritical Bernoulli percolation in
  \( \mathsf{Z}^2 \), and let \( G \) be the connected component of \( \mathcal{C}^c \) containing
  the origin. For \( p\geq 8/9 \),
  \begin{equation*}
    E\,\vert G \vert \leq \alpha(1-p),
  \end{equation*}
  where \( \alpha \) is some positive constant.
\end{lemma}
\begin{proof}
  Suppose \( |G|=n \), \( n\geq 1 \). The isoperimetric inequality implies that there exists
  \( x \in \partial^iG \) with \( \lfloor\sqrt n/2\rfloor\leq |x|_\infty \leq n-1 \).  As
  \( 0 \in G \), there must also exist \( y \in \partial ^iG \) with
  \( |x-y|_\infty\geq \lfloor \sqrt n/2\rfloor \) (in case \( n=1 \), \( x=y \)). Furthermore, as
  \( \partial ^iG \) is a \( |{\,\cdot\,}|_\infty- \)connected set, there must exist a
  self-avoiding path
  \begin{equation*}
    x=x_0, x_1, \dots ,x_k=y,
  \end{equation*}
  composed of consecutive \( |{\,\cdot\,}|_\infty- \)neighbors in \( \partial ^iG \), going from
  \( x \) to \( y \), with \( k\geq \lfloor\sqrt n/2\rfloor \). The probability that such a path,
  starting at \( x \) fixed, exists, is
  \begin{equation*}
    \leq (1-p)[8(1-p)]^k\leq (1-p)[8(1-p)]^{\lfloor\sqrt n/2\rfloor}, 
  \end{equation*}
  since at each step there are 8 possible \( |{\,\cdot\,}|_\infty- \)neighbors, and
  \( 8(1-p)\leq 8/9<1 \). As \( x \) has not more than \( 4n^2 \) possible positions,
  \begin{equation*}
    P\big(\,|G|=n,\,0\in G\,\big)\leq 4n^2(1-p)[8(1-p)]^{\lfloor\sqrt n/2\rfloor}  .
  \end{equation*}
  Therefore, calling \( \alpha \) the sum below for \( p=8/9 \),
  \begin{equation*}
    E\big[\,\vert G \vert,\,0\in G\,\big] \leq (1-p)\sum_{n\geq 1}4n^3[8(1-p)]^{\lfloor\sqrt n/2\rfloor}\leq \alpha (1-p).
  \end{equation*}
\end{proof}

\begin{proof} (of Theorem \ref{thm:perc}, first half: \( 0<u_c(s) \)) Here we present the proof
  assuming \( d=2 \). The proof for \( d\geq 3 \) is the same, we leave the details to the
  reader. Fix \( s>0 \). We will show that there exists a positive \( u \) which belongs to the set
  in (\Ref{eq:lambda-critical}), which proves the assertion.  For this we need an appropriate
  adapted escape strategy which we will construct by induction on the time coordinate \( t
  \). The precise value of such \( u \) will be determined below.

  \par Note that for each \( t\geq 0 \), \( (\xi(x,t))_{x\in \mathsf{Z}^2} \) is a two dimensional Bernoulli site
  percolation configuration with probability of any site \( x \) being open
  \( P^u (\xi(x,t)=0)=e^{-u} \). Consider also, for each \( t \), the Bernoulli site percolation
  processes with sites given by \( 3\times 3 \) squares \( B(x,1) \), \( x\in 3\,\mathsf{Z}^2 \),
  and \( 9\times 9 \) squares \( B(x, 4) \), \( x\in 9\,\mathsf{Z}^2 \), with probabilities of a
  site being open \( P^u(\xi(y,t)=0 \text{ for all } y\in B(x,1))=e^{-9u} \), and
  \( P^u(\xi(y,t)=0 \text{ for all } y\in B(x,4))=e^{-81u} \), respectively. Let us refer to these
  squares as 3-squares and 9-squares. These percolation processes are not independent across time
  since traps move to neighboring sites at each time transition.

  \par Let \( u=u_0>0 \) be small enough to have percolation of 9-squares. Call \( C(t),C_3(t) \) and
  \( C_9(t) \) the unique infinite clusters of open sites and of open 3-squares and 9-squares at
  time \( t \), respectively, the three viewed as subsets of \( \mathsf{Z}^2 \). It is clear that
  \( \emptyset \neq C_9(t) \subset C_3(t) \subset C(t) \) for all \( t \) with probability 1.

  \par We now construct the escape strategy \( \rho \). First don't move:
  \( \rho_0=(x_0,t_0)=(0,0) \), \( \rho_1=(x_1,t_1)=(0,1) \). Then, at time \( t=1 \) look at
  \( C_3(t-1)=C_3(0) \): either \( x_1=0 \) is the center of a 3-square in \( C_3(0) \), or
  \( x_1 \not\in C_3(0) \).  In this last case stay at \( 0 \) forever: \( \rho_k=(0,k) \) for all
  \( k\geq 2 \). We describe the alternative case in a generic way.

  In case \( x_k \) is the center of a 3-square in \( C_3(t-1) \) at time \( t=t_k \), let
  \( (\sigma_0, \dots ,\sigma_l) \) be any shortest path of neighboring sites contained in its
  interior \( C_3(t-1)^\circ \) going from \( x_k \) to the center of the nearest 9-square in
  \( C_9(t-1) \). Such a path must exist, since \( C_9(t-1)\subset C_3(t-1) \), and
  \( C_3(t-1)^\circ \) is connected and intersects every 3-square in \( C_3(t-1) \).  Define
  \( \rho_{k+i}=(\sigma_i,t) \) for \( i=0, \dots ,l \), and \( \rho_{k+l+1}=(\sigma_l,t+1)
  \). Observe now that,
  \begin{enumerate}[(1)]
  \item \( C_3(t-1)^\circ \) is an infinite connected set. And since traps can only move to
    neighboring sites, it is a set of open sites at time \( t \).  It follows that
    \( \sigma \subset C_3(t-1)^\circ \subset C_1(t) \).  This guarantees the survival of the target
    particle at time \( t \).
  \item Similarly, the 3-squares contained in \( C_9(t-1)^\circ \) form an infinite connected set; it is
    therefore included in \( C_3(t) \).  Then, the first space position at time \( t+1 \),
    \( \sigma_l \), being the center of a 9-square in \( C_9(t-1) \), is also the center of a
    3-square in \( C_3(t) \). Thus, the construction of \( \rho \) can be continued inductively,
    and so the target survives forever.
  \end{enumerate}
  In conclusion, in the event that \( C_3(0) \) contains the origin, the escape strategy just
  described guarantees that the target particle survives forever. It follows that for
  \( u \leq u _0 \),
  \begin{equation}
    \label{eq:survival-time}
    0 \,<\, P^{u_0}(0\in C_3(0)) \,\leq\, P^u(0\in C_3(0)) \,\leq\, P^u(T_{\rm det}=\infty),
  \end{equation}
  as \( P^u(0\in C_3(0)) \) is non-increasing in \( u \). On the other hand, the speed \( S_t \) of the
  target particle at time \( t \) is either \( 0 \) or equal to the length \( l \) of \( \sigma \)
  in the above construction.  Let \( G \) be the set of 9-squares in the connected component of
  \( C_9(t-1)^c \) containing \( \sigma_0 \). Then, for
  \( u \leq \min\{u _0,\frac{1}{81}\log\frac{9}{8}\} \),
  \begin{equation*}
    \begin{split}
      E^u[\,S_t,\, T_{\rm det}=\infty\,]
      &\leq  E^u[\,S_t,G=\emptyset \,] +
        E^u[\,S_t, G \neq \emptyset,0\in C_3(0)\,] +
        E^u[\,S_t, 0 \not\in C_3(0)\,] \\
      &\leq 81\alpha (1-e^{-81u}) + 9P^u(\,G \neq \emptyset\,),
    \end{split}
  \end{equation*}
  where we used Lemma \Ref{lm:percolation-component} and the fact that \( S_t=0 \) in case \( G \)
  is empty or \( 0 \not\in C_3(0) \), and otherwise \( S_t \) is not larger than \( 81|G|+9 \);
  notice that if \( \sigma_0\in C_9(t-1) \), the target doesn't move, for it is already placed at
  the center of a 9-square in \( C_9(t-1) \). Finally, as \( P^u (G=\emptyset) \leq E^u |G| \),
  using (\Ref{eq:survival-time}) we get
  \begin{equation*}
    E^u[\,S_t\mid T_{\rm det}=\infty\,] \leq \frac{90\alpha(1-e^{-81u})}{P^{u_0}(0\in C_3(0))},
  \end{equation*}
  which can be made as small as desired taking \( u \) small enough.  The existence of a positive
  \( u \) satisfying the conditions in the defining set of \( u _c(s) \) is proved, and thus also
  the theorem.
\end{proof}
\begin{remark}
  \label{rem:coupling}
  If we drop the condition of the escape strategy being adapted in the definition
  (\Ref{eq:lambda-critical}) of \( u_c(s) \) its value might increase, call it
  \( \tilde{u}_c(s) \). Later on we construct again the trap environment in the context of random
  interlacements via a coupling (see \Ref{eq:random-interlacements}). That construction makes it
  clear that the set of escape strategies in the definition of \( u_c(s) \), or
  \( \tilde{u}_c(s) \) for that matter, is decreasing in \( u \). It is also clear that it is
  increasing in \( s \). It follows that in order to prove the second half of Theorem
  \Ref{thm:perc} it is enough to prove that there exits a \( u<\infty \) such that no escape
  strategy, adapted or not, exists with \( P^u (T_{\rm det}=\infty)>0 \). In Theorem \Ref{thm:main}
  we prove something stronger in fact, that there is no usual percolation of the vacant set left by
  the traps.
\end{remark}

\section{Random Interlacements formulation}

\par We now formulate the escape problem in the framework of Random Interlacements, which allow us
to prove the second half of Theorem \Ref{thm:perc} by using the heavy renormalization machinery
already developed in the context of Random Interlacements. The Random Interlacements process is a
random subset of \( \mathsf{Z}^d \) formed by a collection of random walks randomly located in
space as points of a Poisson point process. We need to define a modification of the Random
Interlacements process where the random walks are not symmetric but have a preferred direction of
motion. The original Random Interlacements model was generalized to include random walks on
transient weighted graphs \cite{teixeira2009interlacement} but the corresponding random walk is
reversible, which ours is not.

\par We need to define first the set of upward jumps
\begin{equation*}
  J = J(d)= \{(a,1): |a|_1=1, a\in \mathsf{Z}^{d-1}\},
\end{equation*}
and allude to the negatives of elements in \(J\) as downward jumps. And while we're at it, we also
agree to refer to the last coordinate as the vertical coordinate or the time dimension.  Define now
the space of infinite upward paths
\begin{equation*}
  W=\{w:\mathsf{Z}\to \mathsf{Z}^d\;:\;w(n+1)=w(n)+e,\, e\in J,\,n\in \mathsf{Z}\},
\end{equation*}
and similarly upward and downward paths starting at time 0
\begin{equation*}
  W_{\pm}=\{w:\mathsf{N}\to \mathsf{Z}^d\;:\;w(n+1)=w(n)\pm e, \,e\in J,\,
  n\in \mathsf{N}\}.
\end{equation*}
Endow this sets with the product sigma algebra, denoted by \(\mathcal{W}\) and \(\mathcal{W}_\pm\)
respectively, generated by the coordinate random variables \(X_n\), for \( n\in \mathsf{Z} \) or
\( n\in \mathsf{N} \), respectively.

\par Given \(x\in \mathsf{Z}^d\), \( d\geq 3 \), define the upward random walk starting at \(x\)
with Markovian law \(P_x^+\) on \(W_+\) satisfying \(P_x^+(X_0=x)=1\) and
\(P_x^+(X_{n+1}=z\vert X_{n}=y)=p(y,z)\), where the transition probabilities are
\begin{equation*}
  p(y,z)=\left\{ 
    \begin{array}{l@{\quad:\quad}l}
      \frac{1}{2(d-1)} & z=y+e, e\in J\\
      0 & \text{otherwise.} 
    \end{array}\right.
\end{equation*}
Similarly, on \(W_-\) we define the downward paths law \(P^-_x\) by changing the transition
probabilities to
\begin{equation*}
  p(y,z)=\left\{ 
    \begin{array}{l@{\quad:\quad}l}
      \frac{1}{2(d-1)} & z=y-e, e\in J\\
      0 & \text{otherwise.} 
    \end{array}\right.
\end{equation*}

\par Let's define also the hitting time of \( K \), \( \widetilde{H}_K=\inf\{n\ge 1: X_n\in K\} \), and the
entrance time \( {H}_K=\inf\{n\ge 0: X_n\in K\} \). On \( W \), on the other hand, we will
distinguish the paths entering a set \( K\subset \mathsf{Z}^d \) at time \( n \),
\( W_K^n=\{w\in W: H_K(w)=n\} \), where \( H_{K}=\inf\{n\in \mathsf{Z}: X_n\in K\} \) is the
entrance time of \( K \). The union \( W_K=\cup_{n\in\mathsf{Z}}W_K^n \) is the set of paths in
\( W \) entering \( K \).  The time shift operators \( \theta_k \) on \( W \) for
\( k\in \mathsf{Z} \), or on \( W_\pm \) for \( k\in \mathsf{N} \), are defined as
\( \theta_k(w)(n)=w(k+n) \). To indicate that a subset \( K \subset \mathsf{Z}^d \) is finite we
will write \( K\subset \subset \mathsf{Z}^d \).  Occasionally we will use the notation
\( X\sim Y \) to indicate that random variables \( X \) and \( Y \) have the same distribution. We
also define, for \( m \) a measure in \( \mathsf{Z}^d \), the measure
\( P_m^+=\sum_{x} m(x)P_x^+ \) defined in \( (W_+, \mathcal{W}_+) \). The distance between \( A \)
and \( B \), subsets of \( \mathsf{Z}^d \), is defined as \( d(A,B)=\min\{|x-y|_\infty :x\in A,y\in B\} \)
and the radius of \( A \) as \( \mathrm{rad}(A)=\min\{r:\exists\, B(z,r)\supset A\} \). Finally, the set of
upward path excursions from \( A \) to \( B \) is denoted as
\begin{equation}\label{eq:excursion-def}
  \begin{split}
    \mathcal{C}_{A,B}=\cup_{n\geq 1}\big\{\sigma =(\sigma_k
    &)_{0\leq k\leq n}:\, \sigma_0\in A, \sigma_n\in B, \\
    & \quad \;\sigma_k\not\in B\text{ and }\sigma_{k+1}-\sigma_k\in J\text{ for }0\leq k<n \big\}.
  \end{split}
\end{equation}
Notice that \( \mathcal{C}_{A,B} \) might be empty even if \( A \) and \( B \) are not.

\par The problem of site percolation we are interested in, has an open site in
\( x\in \mathsf{Z}^d \) if no path passes through it and a closed site if any path does. The
probability distribution of the paths appearing in a given realization is introduced right after
Theorem \ref{thm:exist}. As we need to deal with events defined by subsets of \( \mathsf{Z}^d \), let us
introduce the measurable space of those subsets, which we will also call \emph{configurations},
\begin{equation*}
  (\{0,1\}^{\mathsf{Z}^d}, \mathcal{F}), \quad  \mathcal{F} = \sigma (\Psi_x:x\in\mathsf{Z}^d),
\end{equation*}
where the canonical projection map \( \Psi_x : \{0,1\}^{\mathsf{Z}^d}\to\{0,1\}\) is defined as
\( \Psi_x(\xi)=\xi(x) \) for \( x\in \mathsf{Z}^d \). Given a configuration
\( \xi \in \{0,1\}^{\mathsf{Z}^d} \), the associated subset \( \mathcal{I}\subset \mathsf{Z}^d \)
is given by
\begin{equation*}
  \mathcal{I}=\{x\in \mathsf{Z}^d:\xi _x=1\}.
\end{equation*}
Reciprocally, \( \mathcal{I}\subset \mathsf{Z}^d \) defines the configuration given by
\( \xi =(1_{\mathcal{I}}(x))_{x\in \mathsf{Z}^d} \). An event \( G\in \sigma (\Psi_x:x\in K) \), depending
only on a finite set of coordinates \( K\subset\subset\mathsf{Z}^d \), is said to be a \emph{local
  event} with support \( K \). The space of configurations \( \{0,1\}^{\mathsf{Z}^d} \) can be
partially ordered in the following way:
\begin{equation*}
  \xi \leq \xi'\quad \text{ if } \quad  \xi (x) \leq \xi '(x) \text{ for all }x\in \mathsf{Z}^d
\end{equation*}
We say an event \( G\in \mathcal{F} \) is \emph{increasing} if \( \xi \leq \xi ' \) and
\( \xi \in G \) imply \( \xi' \in G \), and \emph{decreasing} if \( \xi \leq \xi ' \) and
\( \xi' \in G \) imply \( \xi \in G \).

\par In analogy with the case of simple symmetric random walk (see Chapter 1 of
\cite{drewitz2014introduction}), we define an ``directed'' Green function,
\begin{equation*}
  g(x,y) = E_x^+\bigg[\sum_{n\geq 0} 1_{\{X_n=y\}}\bigg] = P_x^+(H_y<\infty).
\end{equation*}
The last equality follows from the fact that \( P^+_x(X_n=y)= 0 \) for \( n\neq y_d-x_d \), the time
difference between \( x \) and \( y \). The \emph{capacity} of \( K \) is now defined as
\begin{equation*}
  \mathrm{cap}(K)=\sum_{x\in\mathsf{Z}^d} e_K(x),
\end{equation*}
where
\begin{equation*}
  e_K(x) = P_x^-\big(\widetilde{H}_K=\infty\big)1_{\{x\in K\}},
\end{equation*}
is the \emph{equilibrium measure} of \( K \). Notice its support \( \mathrm{supp}(e_K) \) is included in
\( \partial ^i(K) \), the inner boundary of \( K \). We also define the normalized equilibrium
measure
\begin{equation*}
  \tilde{e}_K(x)= \frac{e_K(x)}{\mathrm{cap}(K)}.
\end{equation*}
There is a number of properties satisfied by these functions, much in analogy to the classical
theory, and also some important differences. For example, the capacity of an ``spatial'' set is
just its cardinality
\begin{equation}
  \label{eq:spatial-set-capacity}
  \mathrm{cap}(A\times\{t\})=|A| \quad \text{for every} \quad t\in \mathsf{Z}, A\subset \mathsf{Z}^{d-1}.
\end{equation}
Later we will use the fact that capacity is time symmetric
\begin{lemma}
  \label{lm:capacity-time-reverse}
  \begin{equation*}
    \mathrm{cap}(K) = \mathrm{cap}(K\!\!\updownarrow) = \sum_{x\in K} P_x^+(\widetilde{H}_K=\infty),
  \end{equation*}
  where \( K\!\!\updownarrow =\{(x_1, \dots , x_{d-1}, -x_d): x\in K\}\) is the time reflected set of \( K \).
\end{lemma}
\begin{proof}
  Let \( t'=\min\{x_d: x\in K\} \), \( t''=\max\{x_d: x\in K\} \). Similarly define the time level sets
  \( L'=\{t'\}\times \mathsf{Z}^{d-1} \), \( L''=\{t''\}\times \mathsf{Z}^{d-1} \). Also for
  \( A,B\in \mathsf{Z}^d \) define the set of upward (downward) paths going from \( A \) to \( B \)
  as
  \begin{equation*}
    \begin{split}
      \mathcal{E}_{A,B}^\pm=\cup_{n\geq 1}\big\{\sigma =(\sigma_k)_{0\leq k\leq n}&\,:\,\sigma_0\in A,\, \sigma_n\in B,\\
        \sigma_{k+1}&\not\in A, \,\sigma_k\not\in B \text{ and } \pm(\sigma_{k+1}-\sigma_k)\in J \text{ for } 0\leq k\leq n\,\big\}
    \end{split}
  \end{equation*}
  We have
  \begin{align*}
    \mathrm{cap}(K) &= \sum_{x\in K}P_x^-\big(\widetilde{H}_K=\infty\big)\\
    &= \sum_{x\in K}\;\sum_{\sigma \in \mathcal{E}^-_{\{x\},L'}} \left(\frac{1}{2(d-1)}\right)^{x_d-t'}
    = \sum_{\sigma \in \mathcal{E}^-_{L'',L'},\; \sigma \cap K\neq\emptyset } \left(\frac{1}{2(d-1)}\right)^{t''-t'}.
  \end{align*}
  In a similar way
  \begin{equation*}
    \sum_{x\in K}P_x^+\big(\widetilde{H}_K=\infty\big) = \sum_{\sigma \in \mathcal{E}^+_{L',L''},\; \sigma \cap K\neq\emptyset } \left(\frac{1}{2(d-1)}\right)^{t''-t'}.
  \end{equation*}
  The sums on the right in both last expressions are equal. This proves the lemma.
\end{proof}
We will also need the following direct application of the Local Central Limit Theorem.
\begin{lemma}
  \label{lm:Green-LCLT}
  There exists \( c_d=c(d)>0 \) such that for \( x,y \) with \( y_d>x_d \),
  \begin{equation*}
    g(x,y)\leq c_d\,(y_d -x_d)^{-\frac{{d-1}}{2}}.
  \end{equation*}
\end{lemma}

\par We want to study the sets defined by the traces of infinite paths in \( W \) appearing as points
of a Poisson point process. To do this properly we introduce an equivalence of paths in \( W \):
\begin{equation*} 
  w\sim w' \quad\text{if}\quad w({\,\cdot\,})=w'(k+{\,\cdot\,}), \;\text{ for some
  } k \in \mathsf{Z}.
\end{equation*}
Call \( W^{*}=W/\sim \) the quotient set, and \( \pi:W\to W^{*} \) the canonical projection. The
\( \sigma \)-algebra \( \mathcal{W}^{*}=\{A\subset W^{*}: \pi^{-1}(A)\in \mathcal{W}\} \) will
allow us to define on \( W^{*} \) the measures we need. Furthermore denote by
\( W_K^{*}= \pi(W_K)=\pi(W_K^n) \), \( n\in\mathsf{Z} \), the set of equivalence classes of paths
entering \( K \).

\par For \( K\subset\subset \mathsf{Z}^d \) let us define the measure \( Q_K \) on \( W \) by the following condition. For
any \( A\in \mathcal{W}_{-} \), \( B\in W_{+} \) and \( x\in K \),
\begin{equation}
  \label{eq:def-Q}
  Q_K\big[(X_{-n})_{n\ge0}\in A, X_0=x, (X_n)_{n\ge0}\in B\big]
  =P^-_x\big(A,\widetilde{H}_K=\infty\big)P_x^+(B).
\end{equation}
Notice it is supported on \(W_K^0\), with total mass
\begin{equation*}
  \sum_{x\in K}Q_K(X_0=x)=\sum_{x\in K}e_K(x)= \mathrm{cap}(K).
\end{equation*}
\begin{theorem}
  \label{thm:exist}
  There is a unique measure \(\nu\) on \((W^{*},\mathcal{W}^{*})\) such that for any
  \( K\subset\subset \mathsf{Z}^d \)
  \begin{equation*}
    1_{W_K^{*}}\nu = Q_K\pi^{-1}.
  \end{equation*}
\end{theorem}

\par We can now construct, following the general scheme presented in \cite{sznitman2010vacant}, the upward
paths random interlacements process we are interested in. Let \(\Omega\) be the set of locally
finite point measures on \( W^*\times \mathsf{R}_+ \) that is, measures of the form
\begin{equation*} 
  \omega = \sum_{n\ge0}\delta_{(w_n^*,u_n)}, \quad w_n^*\in W^*,u_n\in \mathsf{R}_+,
\end{equation*}
satisfying \( \omega(W^*_K\times[0,u])<\infty \) for all \( K\subset\subset\mathsf{Z}^d \),
\( u>0 \), and let \( \mathcal{A} \) be the canonical \( \sigma \)-algebra on \( \Omega \)
generated by the evaluation maps \( \omega\mapsto\omega(D) \),
\( D\in \mathcal{W}^*\otimes\mathcal{B}_{\mathsf{R}_+} \).  \\ We now define on
\( (\Omega,\mathcal{A}) \) a probability measure, call it \( \mathbb{P} \), making the random
element
\begin{equation}
  \label{eq:ppp}
  \omega=\sum_{n\ge0}\delta_{(w_n^*,u_n)}
\end{equation}
a Poisson point process with intensity measure \( \nu\otimes\lambda \). The \emph{random interlacement} at level
\(u > 0\) is now defined as the random variable (random set)
\(\mathcal{I}^u:\Omega\to 2^{\mathsf{Z}^d}\) given by
\begin{align} 
  \label{eq:random-interlacements}
  \mathcal{I}^u(\omega)=\bigcup_{u_n\le u}\mathrm{range}(w_n^*),\quad\text{ for
  }\ \omega=\sum_{n\ge0}\delta_{(w_n^*,u_n)}\in \Omega,
\end{align}
and the associated \emph{vacant set} is just its complement
\( \mathcal{V}^u=\mathsf{Z}^d \setminus \mathcal{I}^u \); here
\( \mathrm{range}(w^*)=\{X_n(w): w\in \pi ^{-1}(w^*), \,n\in \mathrm{Z}\} \) is the set of sites
visited by \( w^* \).

\par As \( \omega(W^*_K\times[0,u]) \) is a Poisson random variable with parameter
\begin{equation*} 
  \mathbb{E}[\omega(W^*_K\times[0,u])]\,=\,(\nu\otimes\lambda)(W^*_K\times[0,u])\,=\,u\, \mathrm{cap}(K),
\end{equation*}
we say that the expected number of interlacement paths passing through \(K\) at level \(u\) is
\(u\,\mathrm{cap}(K)\). It follows that
\begin{equation*} 
  \mathbb{P}[\mathcal{I}^u \cap K=\emptyset ] = \mathbb{P}[\omega(W^*_K\times[0,u])=0]\,=\,e^{-u\,\mathrm{cap}(K)}.
\end{equation*}

\par For \( A_1, \dots ,A_j \subset \subset \mathsf{Z}^{d-1}\times \{0\} \) disjoint,
\( W^*_{A_1}, \dots ,W^*_{A_j} \) are also disjoint. Then the variables equal to the number
\( \omega(W^*_{A_i}\times[0,u]) \) of directed interlacement paths passing through \( A_i \), for
\( i=1, 2, \dots \), is a set of independent Poisson random variables with parameters
\( u \,\mathrm{cap}(A_i)=u|A_i| \), see (\Ref{eq:spatial-set-capacity}). We conclude that the
directed random interlacements model just defined is the negative time extension of the space-time
trap model introduced in Section \Ref{sec:introduction}; the level parameter \( u \) here appeared
as the intensity there. Notice that we have simultaneously defined all random sets
\( \mathcal{I}^u \) and \( \mathcal{V}^u \) for different \( u\in [0,\infty[\) via a coupling,
which implies that for \( u\leq u' \) we have \( \mathcal{I}^u \subset \mathcal{I}^{u'} \)
\( \mathbb{P} \)-almost surely (see Remark \Ref{rem:coupling}). This will be a central feature for
the sprinkling technique we use in order to prove of our phase transition result. Observe that the
occupation events \( A\in \mathcal{F} \), defined before, have a probability
\( \mathbb{P}(\mathcal{I}^u\in A) \) that depends on the level parameter \( u\geq 0 \) used.

\par The next theorem proves the second half of Theorem \Ref{thm:perc} by showing that for large
enough values of the level parameter there is no percolation of the vacant set of upward path
interlacements. To state it we first introduce some notation. Given \( x,y \in \mathsf{Z}^d \),
\( x \overset{\mathcal{V}^u}{\longleftrightarrow} y \) represents the event that
\( \mathcal{V}^u \) contains a path from \( x \) to \( y \) in the lattice, that is, the set
\( \mathcal{V}^u \) contains a sequence of neighboring sites \(x_0\sim x_1 \sim \cdots \sim x_n \)
such that \( x_0=x \) and \( x_n=y \). Define the probability of percolation
\begin{equation*}
  \theta(u)= \mathbb{P}(0\overset{\mathcal{V}^u}{\longleftrightarrow}\infty),
\end{equation*}
where the \emph{percolation event} \( \{0\overset{\mathcal{V}^u}{\longleftrightarrow}\infty\} \) is defined as
\(\cap_{n\geq 1}\{\exists \,x \in \mathsf{Z}^d:|x|_\infty \geq n,\,
0\overset{\mathcal{V}^u}{\longleftrightarrow}x\}\).

\par Since \( \mathcal{V}^{u'} \subset \mathcal{V}^u\) for \( u\leq u' \), \( \theta(u) \) is a decreasing function of
\( u \). This implies that, in order to prove the theorem stated next we just need to show that
there exists a finite value of \( u \) for which \( \theta(u) \) is zero.

\begin{theorem}
  \label{thm:main}
  There exists a positive value \(u_c\) such that
  \begin{equation*}
    \theta (u) = 0 \,\text{ for }\,u>u_c.
  \end{equation*}
\end{theorem}
\section{Proof of Theorems}

\subsection{Existence Theorem}

\begin{proof}
  (of Theorem \ref{thm:exist}) The only difficulty is proving that the condition on
  \(1_{W_K^{*}}\nu\) is not ambiguous (see \cite{drewitz2014introduction}) i.e., that in case
  \(K\subset K'\), we have \(Q_{K'}A = Q_KA\) for \(A=\pi^{-1}B\), \(B\in\mathcal{W}^{*}\) and
  \(B\subset W_K^{*}\) (\(\subset W_{K'}^{*}\) also then).

  Notice that, since \(Q_K\) is supported on \(W_K^0\) while \(Q_{K'}\) is supported on
  \(W_{K'}^0\), for each upward path \(w\in W_{K'}^0\) in the left hand side, we want on the right
  it's time translated equivalent \(\theta_{H_K}(w)\). It is enough to show that for cylinder sets
  \(A=\Pi_{n\in\ZZ}A_n\) (\(A_n\subset \ZZ^d\))
  \begin{equation}
    \label{eq:cylinders}
    Q_{K'}[H_K<\infty, X_{n+H_K}\in A_n, n\in\ZZ]=Q_K[X_n\in A_n, n\in\ZZ].
  \end{equation}
  In fact, we would then have
  \(Q_{K'}[H_K<\infty, (X_{n+H_K})_{n\in\ZZ}\in A]=Q_K[(X_n)_{n\in\ZZ}\in A]\) for all
  \(A\in\mathcal{W}\). In case \(B\subset W_K^{*}\), \(A=\pi^{-1}B\) being composed of entire
  equivalence classes, the event on the left equals \(\{(X_n)_{n\in\ZZ}\in A\}\) and the result
  follows.

  \par As (\ref{eq:cylinders}) involve only upward paths going through \(K\), let us focus our
  attention on the excursions \(\sigma\in \mathcal{C}_{K',K}\) going from \(K'\) to \(K\), see
  (\ref{eq:excursion-def}). Observe \(\mathcal{C}_{K',K}\neq \emptyset \) in this case. Assume also
  that \(\sigma(n)\in A_{n-l}\) for \(n=0,\dots,l\), \( l=l(\sigma) \) denoting its length. We have
  the following finite partitions for these paths (using the notation \(X_l^m=(X_l,\dots,X_m)\)),
  \begin{equation*}
    \begin{gathered}
      W_{K'}^0\cap\{H_K<\infty, X_{n+H_K}\in A_n,
      n\in \mathsf{Z}\}=\bigcup_{\sigma} W_{K'}^0\cap\big\{X_0^l
      =\sigma, X_{n}\in A_{n-l}, n\in \mathsf{Z}\big\},\\
      W_K^0\cap\{X_n\in A_n, n\in \mathsf{Z}\}=\bigcup_{\sigma}
      \big\{H_{K'}=-l, X_{-l}^0=\sigma, X_{n}\in A_n, n\in \mathsf{Z} \big\}.
    \end{gathered}
  \end{equation*}
  It is then enough to prove, for \(\sigma\) as above, that
  \begin{equation*}
    Q_{K'}\big[X_0^l=\sigma, X_n\in A_{n-l}, n\in \mathsf{Z}\big] = Q_K\big[H_{K'}=-l, 
    X_{-l}^0=\sigma, X_n\in A_n, n\in \mathsf{Z}\big].
  \end{equation*}
  Now, according to (\ref{eq:def-Q}), the left hand side equals
  \begin{multline*}
    \label{eq:exist-proof-2}
    Q_{K'}\big[X_0^l=\sigma, X_{n}\in A_{n-l}, n\in \mathsf{Z}\big] \\ =
    P_x^-\big\{X_n\in A_{-n-l}, n\in \mathsf{N}, \widetilde{H}_{K'}=\infty\big\}
    P_x^+\big\{X_0^l=\sigma, X_n\in A_{n-l}, n\in \mathsf{N}\big\}.
  \end{multline*}
  Clearly \(P_x^+\big\{X_0^l=\sigma\big\}=1/(2(d-1))^l =P_y^-\big\{X_0^l=-\sigma\big\}\), where we
  write \(-\sigma\) to mean \((\sigma_l,\dts,\sigma_0)\). Now apply twice the Markov property to
  get
  \begin{align*}
    &P_x^-\{X_n\in A_{-n-l}, n\in \mathsf{N},
      \widetilde{H}_{K'}=\infty\}\;P_x^+\big\{X_0^l=\sigma\big\}\;
      P_y^+\{X_n\in A_n, n\in \mathsf{N}\} \\
    = &P_x^-\{X_n\in A_{-n-l}, n\in \mathsf{N},
        \widetilde{H}_{K'}=\infty\}\;P_y^-\big\{X_0^l=-\sigma\big\}\;
        P_y^+\{X_n\in A_n, n\in \mathsf{N}\}\\
    = &P_y^-\big\{X_0^l=-\sigma, X_n\in A_{-n}, n\in \mathsf{N},
        \widetilde{H}_{K'}\circ\theta_l=\infty\big\}\; P_y^+\{X_n\in A_n, n\in
        \mathsf{N}\}.
  \end{align*}
  As \(\widetilde{H}_{K}=\infty\) can be included due to the fact that \(-\sigma\) doesn't return
  to \(K\), this equals
  \begin{multline*}
    P_y^-\big\{X_0^l=-\sigma, X_n\in A_{-n}, n\in \mathsf{N},
    \widetilde{H}_{K'}\circ\theta_l=\infty,
    \widetilde{H}_K=\infty\big\}\; P_y^+\{X_n\in A_n, n\in \mathsf{N}\}\\
    = Q_K\big[H_{K'}=-l, X_{-l}^0=\sigma, X_n\in A_n, n\in \mathsf{Z}\big].
  \end{multline*}
\end{proof}

\subsection{Decorrelation Inequalities}

We now introduce the decorrelation inequalities, which serve as the fundamental component in the
implementation of the previously mentioned renormalization scheme.

The key step in the pursuit of this goal is the comparison, through coupling, of variables with
different level parameters \( u_- \) and \( u_+ \), the so called sprinkling technique, Theorem
\Ref{thm:coupling}. Let \( S_i \) and \( U_i \), \( i=1,2 \) be subsets of \( \mathsf{Z}^d \) such
that
\begin{equation}
  \label{eq:S-U}
  \begin{gathered}
    \emptyset \neq S_i\subset U_i\subset \subset \mathsf{Z}^d, \;\;\; i=1,2 \quad \text{ and } \quad  U_1\cap U_2=\emptyset, \\
    \text{and let }\quad  S=S_1\cup S_2, \quad  U=U_1\cup U_2.
  \end{gathered}
\end{equation}
We will define by restriction and composition, starting from the random interlacements point
process (\ref{eq:ppp}), \( \omega = \sum_{n\geq 0} \delta_{(w^*_n,u_n)} \), a succession of point
processes that will allow us to look in detail at the upward path occupation events we are
interested in.

\par We want to look at the paths of the interlacement process when they are in \( S \). The idea is
to focus attention on \( U \), which contain \( S_1 \) and \( S_2 \), and show that when this sets
are placed far apart from each other \( U \) can be large, and then use this to show that the
occupation events by \( \mathcal{I}^u \) of \( S_1 \) and \( S_2 \) are weakly correlated.  We
start by defining the Poisson point process on \( W_+ \)
\begin{equation*}
  \mu _{S,u}(\omega) = \sum_{n\geq 0}\delta_{\tau_S(w^*_n)_+}1_{\{w^*_n\in W^*_S,u_n\leq u\}}, \quad \text{ for } \quad \omega=\sum_{n\geq 0}\delta_{(w_n^*,u_n)},
\end{equation*}
where
\begin{equation*}
  \tau_S:W^*_S \ni w^*\mapsto w^0 \in W^0_S
\end{equation*}
chooses the representative of \( w^* \) that enters \( S \) at time \( 0 \), and
\( ({\,\cdot\,})_+:W \ni w\mapsto(w_n)_{n\in \mathsf{N}} \in W_+ \) forgets its negative time
coordinates. Using Theorem \ref{thm:exist} and the fact that \( \tau _S \) is the inverse of
\( \pi\!\mid_{W_S^0} \) we have for the image of \( \nu \) (the random walks part of the intensity
of \( \omega \)),
\begin{equation*}
  (1_{W_S^*}\nu)(\tau_S({\,\cdot\,})_+)^{-1} = \sum_{x\in S}e_S(x)P^+_x = P^+_{e_S}.
\end{equation*}
Therefore the intensity measure of \( \mu _{S,u} = \sum_{n} \delta_{w_n} \) is
\begin{equation*}
  u\,\mathrm{cap}(S)P^+_{\tilde{e}_S}.
\end{equation*}

\par Recall the definition of the entrance time to \( S \), \( H_S \), and define also the exit time
\( T_U = \inf\{ n\geq 0:X_n \not\in U\} \). We want to distinguish paths by their excursions from
the time they enter \( S \) until they exit \( U \). For this let us define recursively the
stopping times
\begin{equation*}
  \begin{gathered}
    R_1=H_S, \quad \quad D_1=T_U\circ\theta _{R_1} + R_1,  \\
    R_{k+1} = R_1\circ\theta_{D_k} + D_k, \quad \quad D_{k+1} = D_1\circ\theta _{D_k} + D_k,
  \end{gathered}
\end{equation*}
with the convention that if any of them takes the value \( +\infty \) the next does too. We now define
new Poisson point processes on the excursions of paths which perform exactly \( j\geq 1 \)
excursions from their entrance time to \( S \) until their exit time from \( U \), and for values
\( 0<u_-<u_+ \) of the level parameter
\begin{equation*}
  \begin{split}
    \zeta_\pm^j
    &= \mu_{S,u_\pm}\phi_j^{-1} \\
    &= \sum_{n} 1_{\{R_j<\infty = R_{j+1}\}}(w_n)\, \delta_{\phi_j(w_n)}, \quad  \quad \text{for} \quad \mu _{S,u_\pm}=\sum_{n}\delta _{w_n},
  \end{split}
\end{equation*}
where the function \( \phi _j \) collect the excursions made by each path,
\begin{equation*}
    \begin{split}
      & \phi_j:\{R_j<\infty = R_{j+1}\}\ni w\mapsto (w^1,\dots,w^j)\in \mathcal{C}_{S,U^c}^j,\\
      & w^k \coloneq (X_n(w))_{R_k(w)\leq n\leq D_k(w)};
  \end{split}
\end{equation*}
recall the definition of \( \mathcal{C}_{A,B} \) in (\ref{eq:excursion-def}). As
\( (\zeta_-^j, \zeta_+^j) \) have disjoint support for different values of \( j \), they are
independent. The intensity measure of \( \zeta_\pm^j \) is given by
\begin{equation}
  \label{eq:intensity-zeta-j}
  \xi_\pm^j (w^1,\dots, w^j) = u_\pm \mathrm{cap}(S)P_{\tilde{e}_S}^+\left\{ R_j<\infty = R_{j+1}, (X_{R_k},\dots,X_{D_k})=w^k, 1\leq k\leq j \right\},
\end{equation}
\( (w^1,\dots, w^j)\in \mathcal{C}_{S,U^c}^j \). Observe that \( \xi_\pm^j \) is equal to zero in many cases, for instance
when the final position of a given excursion makes it impossible for the directed random walk to do
a later excursion.

\par Finally, define the Poisson point process
\begin{equation}
  \label{eq:zeta-star-plus-minus}
  \zeta^*_{-,+}=\zeta^1_+ - \zeta^1_-,
\end{equation}
which will be used to bound the random point measure collecting all excursions from paths returning
at least once to \( S \)
\begin{equation}
  \label{eq:zeta-star-star-minus}
  \begin{split}
    \zeta ^{* *}_-
    &=\sum_{j\geq 2} s_j(\zeta^j_-) \\
    &=\sum_{n} \sum_{j\geq 2} 1_{\{R_j<\infty = R_{j+1}\}}(w_n) \sum_{k=1}^j \delta_{\phi_j(w_n)^k}, 
  \end{split}
\end{equation}
where \( s_j \) is the map taking finite point measures \( \zeta = \sum_{n} \delta_{w_n} \) on
\( \mathcal{C}_{S,U^c}^j \) to finite point measures on \( \mathcal{C}_{S,U^c} \) defined as
\begin{equation*}
  s_j(\zeta) = \sum_{n} \delta_{w_n^1} + \cdots + \delta_{w_n^j}.
\end{equation*}
Observe that \( \zeta^1_- \) and \( \zeta^*_{-,+} \) are independent since they have disjoint
support. For the same reason \( \zeta^1_- \) and \( \zeta ^{* *}_\pm \) are also independent. We now present the
central result where the techniques differ somewhat from the case of Random Intelacements of simple
symmetric random walks. Its  proof is postponed to Section \Ref{sec:coupl-constr}.

\begin{theorem}
  \label{thm:coupling}
  Let \( 0<u_-< u_+=u_-(1+\delta) \). Suppose
  \begin{equation}
    \label{eq:hipot-thm-coupling}
    d(S_1,S_2)>4L, \qquad  \mathrm{cap}(S)\leq \frac{\delta}{4e^4c_d} \,L^{\frac{{d-1}}{2}} \quad \text{and} \qquad \delta \leq e^3,
  \end{equation}
  where \( L=\max\{\mathrm{rad}(S_1),\mathrm{rad}(S_2)\} \), \( S_1, S_2 \) as in (\Ref{eq:S-U}), and
  \( c_d \) defined in Lemma \Ref{lm:Green-LCLT}. Then there exists \( U \) satisfying
  (\Ref{eq:S-U}) and a probability space
  \( (\overline{\Omega}, \,\overline{\!\!\mathcal{A}}, \overline{\mathbb{P}}) \) where it can be
  defined a coupling, \( (\bar{\zeta}_{-,+}^*, \bar{\zeta}_-^{* *}) \) of \( \zeta_{-,+}^* \) and
  \( \zeta _-^{* *} \), such that
  \begin{equation}
    \label{eq:coupling-poisson-ineq}
    \overline{\mathbb{P}} (\bar{\zeta}_{-,+}^* < \bar{\zeta}_-^{* *}) \leq 2e^{-\mathrm{cap}(S)(u_+ -u_-)/20}.
  \end{equation}
\end{theorem}

\par The other ingredient needed to obtain the decorrelation inequalities, Theorem \Ref{thm:decorr},
used in the induction step of the renormalization scheme necessary to prove our main result is
\begin{theorem}
  \label{thm:decorr-coupling}
  Let \( 0<u_-<u_+ \), \( S \) and \( U \) subsets of \( \mathsf{Z}^d \) as in (\Ref{eq:S-U}), and
  \( K_1\subset S_1 \), \( K_2\subset S_2 \) nonempty. If
  \( (\bar{\zeta}_-^{* *}, \bar{\zeta}_{-,+}^*) \) is a coupling of \( \zeta _-^{* *} \) and
  \( \zeta_{-,+}^* \) in some probability space
  \( (\overline{\Omega}, \overline{\!\!\mathcal{A}}, \overline{\mathbb{P}}) \) satisfying
  \begin{equation}
    \label{eq:coupling-ineq}
    \overline{\mathbb{P}} (\bar{\zeta}_-^{**}\leq \bar{\zeta}_{-,+}^*) \geq 1-\epsilon
  \end{equation}
  for some \( \epsilon >0 \), then for \( A_1\in \sigma (\Psi_x : x\in K_1)\), \( A_2\in \sigma (\Psi_x : x\in K_2)\), local events,
  \begin{enumerate}[(a)]
  \item
    \;\( \mathbb{P}(\mathcal{I}^{u_-} \in A_1 \cap A_2) \le \mathbb{P}(\mathcal{I}^{u_+} \in A_1)\, \mathbb{P}(\mathcal{I}^{u_+}\in A_2) + \epsilon\)\; if
    \(A_1,A_2 \) are increasing,
  \item
    \;\( \mathbb{P}(\mathcal{I}^{u_+} \in A_1 \cap A_2) \le \mathbb{P}(\mathcal{I}^{u_-} \in A_1)\, \mathbb{P}(\mathcal{I}^{u_-}\in A_2) + \epsilon\)\; if
    \(A_1,A_2 \) are decreasing.
  \end{enumerate}
\end{theorem}

This theorem is proved in \cite{drewitz2014introduction} Theorem 7.9. As
the reader can verify, the proof doesn't depend on the type of random walk used to define the
random interlacements process. 

\subsection{Renormalization Scheme}

\par We now describe the renormalization scheme used to prove the second half of our main result (see
Remark \Ref{rem:coupling}). We follow the general ideas in \cite{sznitman2012decoupling}, also
detailed in \cite{drewitz2014introduction}.  Consider the geometric scaling of lattice
\( \mathsf{Z}^d\) given by \(\mathcal{L}_n=L_n \mathsf{Z}^d\), where \(L_n=l_0^nL_0\) and
\(L_0 \geq 1\), \(l_0\geq1\) are integers.  Let
\begin{equation*}
  T_n=\cup_{0\le k\le n} T_{(k)} \quad \text{ where } \quad T_{(k)}=\{1,2\}^k,
\end{equation*}
be the canonical binary tree of depth \(n\), and call \emph{proper embedding} of \(T_n\) in
\(\mathsf{Z}^d\) a function \(\mathcal{T}:T_n\to \mathsf{Z}^d\) that satisfies:
\begin{enumerate}[(a)]
\item \(\mathcal{T}(\emptyset)\in\mathcal{L}_n\), \(\emptyset\in T_{(0)}\), the root of the tree,
\item for all \(0\le k\le n\) and \(m\in T_{(k)}\) we have \(\mathcal{T}(m)\in \mathcal{L}_{n-k}\),
\item for all \(0\le k< n\) and \(m\in T_{(k)}\) we have
\begin{equation*}
  \mathcal{T}(m,1),\mathcal{T}(m,2) \in B(\mathcal{T}(m),L_{n-k}) \qquad \text{and}
  \qquad |\mathcal{T} (m,1 ) - \mathcal{T} (m,2 )|_\infty  > \frac{L_{n-k}}{100},
\end{equation*}
\( (m,1) \) and \( (m,2) \) being the two descendants of \( m \) contained in \( T_{(k+1)} \).
\end{enumerate}
Let us denote by \(\Lambda _{x,n}\) the set of proper embeddings of \(T_n\) having root at
\(x\in\mathcal{L}_n\), and by
\begin{equation*}
  G_{\mathcal{T}}=\bigcap_{m \in T_{(n)}}G_{\mathcal{T}(m)}
\end{equation*}
the intersection of events \( G_{\mathcal{T}(m)} \in \mathcal{F} \) indexed by the leaves,
\( m \in T_{(n)} \), of the embedding \( \mathcal{T} \). For instance, the event
\(G_x= \{x \text{ is vacant}\} = \{x \in \mathcal{V}^u\} \) represents site \( x \) being
open. This family of events, for \( x\in \mathsf{Z}^d \), will be used later to prove Theorem
\Ref{thm:main}.

\par Let us also denote as \( \mathcal{T}_1 \) and \( \mathcal{T}_2 \) the two child embeddings
arising from the children of the root of \( \mathcal{T} \) and defined as
\( \mathcal{T}_1(m) =\mathcal{T}(1,m) \) and \( \mathcal{T}_2(m) =\mathcal{T}(2,m) \) for each
\( m \in T_{(k)} \), \( 0\leq k\leq n-1 \). Notice that
\begin{equation*}
  G_{\mathcal{T}} = G_{\mathcal{T}_1} \cap G_{\mathcal{T}_2}.
\end{equation*}

\par The following theorem present the decorrelation inequalities which are the main ingredient in
the renormalization procedure we need in order to prove Theorem \ref{thm:main}.
\begin{theorem}
  \label{thm:decorr}
  Let \( d\geq 3 \) and \(L_0=1\). There exist constants \( C=C(d)<\infty \) and
  \( D=D(d)<\infty \) such that for all \(l_0\geq C\) a multiple of 1000, \(n\geq 0\),
  \( x\in \mathcal{L}_{n+1} \), \(\mathcal{T}\in \Lambda_{x,n+1}\), \(u_->0\), and
  \( G_{\mathcal{T}(m)}\in \sigma (\Psi_y : y\in B(\mathcal{T}(m),2L_0))\), \( m \in T_{(n+1)} \),
  local events,
  \begin{enumerate}[(a)]
  \item if \( G_{\mathcal{T}(m)} \), \( m \in T_{(n+1)} \), are increasing,
    \begin{equation*}
      \mathbb{P}(\mathcal{I}^{u_-} \in G_{\mathcal{T}_1} \cap G_{\mathcal{T}_2}) \le \mathbb{P}(\mathcal{I}^{u_+} \in G_{\mathcal{T}_1}) \, \mathbb{P}(\mathcal{I}^{u_+}\in G_{\mathcal{T}_2}) +
      \epsilon(u_-,n),
    \end{equation*}
  \item if \( G_{\mathcal{T}(m)} \), \( m \in T_{(n+1)} \), are decreasing,
    \begin{equation*}
      \mathbb{P}(\mathcal{I}^{u_+} \in G_{\mathcal{T}_1} \cap G_{\mathcal{T}_2}) \le \mathbb{P}(\mathcal{I}^{u_-} \in G_{\mathcal{T}_1}) \, \mathbb{P}(\mathcal{I}^{u_-}\in G_{\mathcal{T}_2}) +
      \epsilon(u_-,n),
    \end{equation*}
  \end{enumerate}
  where \( u_+=\big(1+\delta(n,l_0)\big)u_- \),
  \begin{equation}
    \label{eq:delta-error}
    \delta(n,l_0)=D (n+1)^{-\frac{3}{2}}l_0^{-\frac{d-2}{4}} \qquad \text{and} \qquad \epsilon (u,n) = 2
    \exp\left(-2u(n+1)^{-3}L_n^{\frac{d-1}{2}}l_0^{\frac{1}{2}}\right).
  \end{equation}
\end{theorem}

\begin{proof}
  The result follows from a direct application of Theorem \Ref{thm:decorr-coupling} with
  \begin{equation*}
    K_1 = \bigcup_{m \in T_{(n)}} B(\mathcal{T}_1(m),2L_0), \quad  K_2 = \bigcup_{m \in T_{(n)}} B(\mathcal{T}_2(m),2L_0).
  \end{equation*}
  To verify its hypothesis (\ref{eq:coupling-ineq}) we
  use Theorem \Ref{thm:coupling} with \( S \) such that
  \begin{equation*}
    \mathrm{rad}(S_i)=\frac{L_{n+1}}{1000} \quad \text{and} \quad K_i \subset S_i \subset B\left(\mathcal{T}(i),\frac{L_{n+1}}{1000}\right), \quad
    i=1,2, 
  \end{equation*}
  and
  \begin{equation}
    \label{ec:bounds-cap-S}
    \frac{1}{2}(n+1)^{-\frac{3}{2}}l_0^{\frac{d}{4}}L_n^{\frac{d-1}{2}} \leq  \mathrm{cap}(S) \leq
    2(n+1)^{-\frac{3}{2}}l_0^{\frac{d}{4}}L_n^{\frac{d-1}{2}},
  \end{equation}
  and \( \delta=\delta(n,l_0) \). That such an \( S \) exists is a consequence of the simple bounds
  \begin{equation*}
    \mathrm{cap}(K)\leq 2^{n+1}2d(4L_0+1)^{d-1}, \quad \quad \bigg(\frac{L_{n+1}}{1000}\bigg)^{d-1}\leq
    \mathrm{cap}\left(B\bigg(\mathcal{T}(1),\frac{L_{n+1}}{1000}\bigg)\right) , 
  \end{equation*}
  the fact that \( \mathrm{cap}(S) \) is non-decreasing and increases not more than 1 with each vertex
  added to \( S \), and that \( l_0 \geq 1000\) and \( L_0=1 \). It remains to show that \( S \)
  satisfy (\Ref{eq:hipot-thm-coupling}) and that
  \( 2e^{-\mathrm{cap}(S)(u_+ -u_-)/20} \leq \epsilon (u_-,n) \). The first inequality of
  (\Ref{eq:hipot-thm-coupling}) is clear, for the second (\Ref{ec:bounds-cap-S}) gives that it is
  enough to satisfy
  \begin{equation*}
    2\,l_0^{\frac{d}{4}}L_n^{\frac{d-1}{2}} \leq \frac{D\,
      l_0^{-\frac{d-2}{4}}}{4e^4c_d} \,\left(\frac{L_{n+1}}{1000}\right)^{\frac{{d-1}}{2}}= \frac{D\,
      l_0^{\frac{d}{4}}}{4e^4c_d} \,\left(\frac{L_n}{1000}\right)^{\frac{{d-1}}{2}};
  \end{equation*}
  We just need to take \( D(d) \geq 8e^4c_d 10^{3(d-1)/2} \).  By (\Ref{eq:delta-error}) and
  (\Ref{ec:bounds-cap-S}),
  \begin{equation*}
    \begin{split}
      2e^{-\mathrm{cap}(S)(u_+ -u_-)/20}
      &\leq 2\exp\left(-\frac{1}{40}(n+1)^{-\frac{3}{2}}l_0^{\frac{d}{4}}L_n^{\frac{d-1}{2}}\,u_-\delta\right)\\
      &= 2\exp\left(-\frac{D}{40}u_-(n+1)^{-3}l_0^{\frac{1}{2}}L_n^{\frac{d-1}{2}}\right) \leq  \epsilon(u_-,n),
    \end{split}
  \end{equation*}
  if we take also \( D(d) \geq 80 \). Finally, for the third inequality of
  (\Ref{eq:hipot-thm-coupling}) we need to choose \( C(d)\geq (D(d)e^{-3})^{\frac{4}{d-2}} \).
\end{proof}

\subsection{Percolation Theorem}

Now the proof of Theorem \Ref{thm:main} is straightforward, as it follows the same steps in
\cite{drewitz2014introduction}:
\begin{enumerate}
\item Use Theorem \Ref{thm:decorr} to obtain a new version of Theorem 8.5 of
  \cite{drewitz2014introduction}. The proof is the same, the only difference is that our function
  \begin{equation*}
    \epsilon (u,n)=2\exp\left(-2u(n+1)^{-3}L_n^{\frac{d-1}{2}}l_0^{\frac{1}{2}}\right),
  \end{equation*}
  looks a little different from theirs,
  \begin{equation*}
    \epsilon (u,n)=2\exp\left(-2u(n+1)^{-3}L_n^{d-2}l_0^{\frac{d-2}{2}}\right),
  \end{equation*}
  but this does not impact the argument of the proof.
\item Use this new theorem to obtain the following version of Theorem 8.7 of
  \cite{drewitz2014introduction}. In order to formulate it we need first to define the function
  \begin{equation}
    \overline{\epsilon}(u,l) = \frac{2e^{-ul^{\frac{1}{2}}}}{1-e^{-u l^{\frac{1}{2}}}}\,,
  \end{equation}
  and, for a family of \( \mathcal{F} \)-measurable events \( G=(G_z)_{z\in \mathsf{Z}^d} \), the event
  \( A(x,N,G) \) that there exists a path \( z_0, \dots , z_k \) in \( \mathcal{L}_0 \) with
  \( |z_i -z_{i-1}|_\infty=L_0 \), also called a \textasteriskcentered-path in \( \mathcal{L}_0 \),
  going from \( x \) to \( B(x,N)^c \) such that the event \( G_{z_i} \) occurs at every vertex
  \( z_i \) of the \textasteriskcentered-path, that is
  \begin{equation*}
    A(x,N,G) =\bigcup_{\substack{z_0, \dots , z_k \in  \mathcal{L}_0 \\z_0=x,\, |x-z_k|>N \\ |z_i-z_{i-1}|=L_0, \, 1\leq i\leq k, \, k\geq 1}}
    \bigcap_{i=0}^k G_{z_i}\, .
  \end{equation*}
  \begin{theorem}
    \label{thm:long-star-paths}
    Let \( d\geq 3 \), \(L_0=1\), and constants \( C=C(d) \) and \( D=D(d) \) as in Theorem
    \Ref{thm:decorr}. Then for all \(l_0\geq C\) a multiple of 1000, \(u>0\), \( \Delta \in\, ]0,1[ \),
    and shift invariant events \( G_y\in \sigma (\Psi_z : z\in B(y,2L_0))\),
    \( y\in\mathsf{Z}^d \), if the following conditions are satisfied,
    \begin{equation}
      \label{eq:hyp-long-star-paths}
      \prod_{n\geq 0}\big(1+\delta(n,l_0)\big) <1+\Delta , \quad \quad (2l_0+1)^{2d}\big[\mathbb{P}(\mathcal{I}^u\in G_0)+ \overline{\epsilon}(u(1-\Delta),l_0)\big] \leq
      e^{-1}, 
    \end{equation}
    (\( \delta(n,l_0) \) defined in (\Ref{eq:delta-error})) then there exists a constant
    \( c =c(u,\Delta,l_0)<\infty \) such that
    \begin{enumerate}[(a)]
    \item if \( G_x \), \( x\in\mathsf{Z}^d \), are increasing, then for any \( N\geq 1 \),
      \begin{equation*}
        \mathbb{P}\left(\,\mathcal{I}^{u(1-\Delta)} \in A(0,N,G)\,\right) \le ce^{-N^{\frac{1}{c}}}, \quad  \text{and}
      \end{equation*}
    \item if \( G_x \), \( x\in\mathsf{Z}^d \), are decreasing, then for any \( N\geq 1 \),
      \begin{equation*}
        \mathbb{P}\left(\,\mathcal{I}^{u(1+\Delta)} \in A(0,N,G)\,\right) \le ce^{-N^{\frac{1}{c}}}.
      \end{equation*}
    \end{enumerate}
  \end{theorem}
  The proof is the same presented in \cite{drewitz2014introduction}.
\item Use this theorem to prove our main result, Theorem \Ref{thm:main}, as follows.
\end{enumerate}
\begin{proof}
  (of Theorem \Ref{thm:main}) We want to show that there exists some finite \( u' \) for which
  \( \mathbb{P}(0\overset{\mathcal{V}^{u'}}{\longleftrightarrow}\infty)=0 \).  Let us take
  \( L_0=1 \). Since a nearest neighbor path is also a \textasteriskcentered-path in
  \( \mathcal{L}_0 \) we have
  \begin{equation*}
    \{0\overset{\mathcal{V}^{u'}}{\longleftrightarrow}\infty\} \subset \{0\overset{\mathcal{V}^{u'}}{\longleftrightarrow}B(0,N)^c\} \subset \{\mathcal{I}^{u'}\in A(0,N,G)\},
  \end{equation*}
  for every \( N\geq 1 \) and \( G_z=\{\xi_z=0\} \), \( z\in \mathsf{Z}^d \). An application of Theorem
  \Ref{thm:long-star-paths} would give us the desired result after taking the limit \( N\to\infty \),
  \begin{equation*}
    \mathbb{P}(0\overset{\mathcal{V}^{u'}}{\longleftrightarrow}\infty) \leq \lim_{N\to\infty}\mathbb{P}(\mathcal{I}^{u'}\in A(0,N,G)) \leq \lim_{N\to\infty}ce^{-N^{\frac{1}{c}}}=0.
  \end{equation*}

  \par The events \( G_z \) are decreasing, thus if we verify conditions
  (\Ref{eq:hyp-long-star-paths}) for some \( u \), \( \Delta \) and \( l_0 \) then 
  \( u'=u(1+\Delta) \). Take \( \Delta =1 /2 \). As \( \delta(n,l_0)\geq 0 \) for all \( n \), 
  \begin{equation*}
      \prod_{n\geq 0}[1+\delta(n,l_0)] \leq \exp\bigg( D\,l_0^{-\frac{d-2}{4}} \sum_{n\geq 0} (n+1)^{-\frac{3}{2}} \bigg),
  \end{equation*}
  implying that for \( l_0 = l_0^* \) sufficiently large the first inequality in
  (\Ref{eq:hyp-long-star-paths}) is satisfied. For the second inequality observe that
  \begin{equation*}
    \begin{split}
    & \mathbb{P}(\mathcal{I}^u\in G_0) = \mathbb{P}(0 \in \mathcal{V}^u) = e^{-u \mathrm{cap}(0)}=e^{-u} \quad  \quad \text{and} \\
    & \overline{\epsilon}(u(1-\Delta),l_0^*) = \overline{\epsilon}\Big(\frac{u}{2},l_0^*\Big) =
    \frac{2e^{-\frac{u}{2} \sqrt{l_0^*}}}{1-e^{-\frac{u}{2} \sqrt{l_0^*}}}\,,   
    \end{split}
  \end{equation*}
  both decreasing in \( u \). Then, for \( l_0 =l_0^* \) and \( u \) large enough both conditions 
  (\Ref{eq:hyp-long-star-paths}) are satisfied and therefore \( u'=\frac{3u}{2} \) is the sought
  for level parameter for which there is no percolation.
\end{proof}

\subsection{Coupling Construction}
\label{sec:coupl-constr}

\begin{proof}
  (of Theorem \Ref{thm:coupling}) We follow the general procedure presented in
  \cite{drewitz2014introduction}.  By the first condition in (\Ref{eq:hipot-thm-coupling}) there
  exist \( z_1,z_2\in \mathsf{Z}^d \) such that
  \begin{equation*}
    W_1=B(z_1,L)\supset S_1, \quad  W_2=B(z_2,L)\supset S_2 \quad \text{and} \quad  |z_1-z_2|>6L.
  \end{equation*}
  We also define for \( i=1,2 \), \( V_i=B(z_i,2L) \), \( U_i=B(z_i,3L) \), and
  \( W=W_1\cup W_2 \), \( V=V_1\cup V_2 \) and \( U=U_1\cup U_2 \). It follows that
  \( U_1\cap U_2=\emptyset \). This is the sufficiently large set \( U \) surrounding \( S \) we
  need in order to make the decorrelation work.

  Our first task is to establish a bound on the intensity measures
  of \( \zeta^*_{-,+} \) and \( \zeta^j_- \). Denote the first as \( \xi^1_{-,+} \). The second,
  \( \xi^j_- \), was already defined in (\Ref{eq:intensity-zeta-j}). We also denote as \( \Gamma \)
  the law of the excursions in \( \mathcal{C}_{S, U^c} \) defined by \( P_{\tilde{e}_S}^+ \), that
  is
  \begin{equation}
    \label{eq:excursions-law}
    \Gamma (w) = P_{\tilde{e}_S}^+[(X_0,\dots X_{T_U})=w], \quad  w\in \mathcal{C}_{S,U^c}.
  \end{equation}
  
  \begin{lemma}
    \label{lm:indep-excursions-ineq}
    \begin{equation*}
      \label{sec:intensity-inequality}
      \xi^1_{-,+} \geq \frac{1}{2}(u_+ - u_-) \, \mathrm{cap}(S)\,\Gamma \quad  \text{ and } \quad
      \xi^j_- \leq \left(\frac{\delta}{2e^4}\right)^{j-1} u_-\, \mathrm{cap}(S)\,\Gamma^{\otimes j}, \quad j\geq 2.
    \end{equation*}
  \end{lemma}
  \begin{remark}
    In fact, the geometry of the sets \( S \) and \( U \) do not allow more than two excursions,
    but we prove the last inequality for every \( j\geq 2 \) anyway, as there is almost no
    aditional difficulty.
  \end{remark}
  \begin{proof}
    As \( \zeta^*_{-,+}=\zeta^1_+ - \zeta^1_- \), by (\ref{eq:intensity-zeta-j}) we have
    \begin{equation*}
      \xi^1_{-,+}(w)= (\xi^1_+ - \xi^1_-)(w) = (u_+ - u_-) \,\mathrm{cap}
      (S)\,P_{\tilde{e}_S}^+[\,R_2=\infty, (X_0,\dots,X_{D_1})=w\,].
    \end{equation*}
    To obtain the first inequality apply the strong Markov property to get
    \begin{equation*}
      \begin{split}
        P_{\tilde{e}_S}^+[\,R_2=\infty, (X_0,\dots,X_{D_1})=w\,]
        &= P_{\tilde{e}_S}^+[\,H_S\circ\theta_{T_U}=\infty, (X_0,\dots,X_{T_U})=w\,]\\
        &\geq P_{\tilde{e}_S}^+[\,(X_0,\dots,X_{T_U})=w\,] \min_{x\in \partial^o U} P^+_x(H_S=\infty)\\
        &= (1-\max_{x\in \partial^o U}P^+_x(H_S<\infty))\, \Gamma (w).
      \end{split}
    \end{equation*}
    Noting that \( S \) is finite, we decompose the event \( \{ H_S<\infty \} \) by the last visit to
    \( S \) and then apply the Markov property to get
    \begin{equation}
      \label{eq:1-beta}
      \begin{split}
        P_x^+(H_S<\infty)
        &= \sum_{y\in S}\, \sum_{n\geq 0}\, P_x^+(X_n=y,X_{n+1} \not\in S, X_{n+2}\not\in S, \dots) \\
        &= \sum_{y\in S}\, \sum_{n\geq 0}\, P_x^+(X_n=y)\, P_y^+(\widetilde{H}_S = \infty) \\
        &=\; \sum_{y\in S} g(x,y)\,P_y^+(\widetilde{H}_S=\infty) \\
        &\leq  \max_{y\in S}g(x,y) \;\mathrm{cap}(S) \\
        &\leq \frac{1}{2} ,
      \end{split}
    \end{equation}
    where we used Lemma \Ref{lm:capacity-time-reverse} in next-to-last line, and in the last line
    the hypothesis (\Ref{eq:hipot-thm-coupling}) in conjunction with Lemma \Ref{lm:Green-LCLT} and
    \( d(S,U^c)>2L \) to obtain
    \begin{equation*}
      \max_{x\in \partial^oU, y\in S}g(x,y) \leq c_d (2L)^{-(d-1)/2}.
    \end{equation*}
    This proves the first inequality.

    \par By (\ref{eq:intensity-zeta-j}), for the second inequality we just have to show
    \begin{equation*}
      P_{\tilde{e}_S}^+\big[ R_j<\infty = R_{j+1}, (X_{R_k},\dots,X_{D_k})=w^k, 1\leq k\leq j \big] \leq \left(\frac{\delta}{2e^4}\right)^{j-1}
      \prod_{k=1}^j \Gamma(w^k),
    \end{equation*}
    or better,
    \begin{equation*}
      P_{\tilde{e}_S}^+\big[ R_j<\infty, (X_{R_k},\dots,X_{D_k})=w^k, 1\leq k\leq j \big] \leq
      \left(\frac{\delta}{2e^4}\right)^{j-1} \prod_{k=1}^j \Gamma(w^k),
    \end{equation*}
    for \( j\geq 1 \). We proceed by induction on \( j \). In case \( j=1 \), \( R_1=0 \)
    \( P_{\tilde{e}_S}^+ \)-almost surely, and then
    \begin{equation*}
      P_{\tilde{e}_S}^+\big[ R_1<\infty, (X_{R_1},\dots,X_{D_1})=w^1\big] = P_{\tilde{e}_S}^+\big[(X_0,\dots,X_{T_U})=w^1\big] = \Gamma (w^1).
    \end{equation*}
    Now assume \( j\geq 2 \). By the strong Markov property at time \( D_{j-1} \) and the induction
    hypothesis
    \begin{equation*}
      \begin{split}
        P_{\tilde{e}_S}^+\big[\, R_j<\infty, \; &(X_{R_k},\dots,X_{D_k}
        )=w^k, 1\leq k\leq j \,\big] \\
        & \!\leq P_{\tilde{e}_S}^+\big[ R_{j-1}<\infty,
          (X_{R_k},\dots,X_{D_k})=w^k, 1\leq k\leq j-1 \big] \\
        & \phantom{\leq P_{\tilde{e}_S}^+\big[ R_{j-1}<\infty}
          \max_{x \in \partial^oU} P_x^+\big[ R_1<\infty,
          (X_{R_1},\dots,X_{D_1})=w^j \big] \\
        & \!\leq \left(\frac{\delta}{2e^4}\right)^{j-2}\,  \prod_{k=1}^{j-1} \Gamma(w^k)
          \, \max_{x \in \partial^oU} P_x^+\big[ R_1<\infty,
          (X_{R_1},\dots,X_{D_1})=w^j \big].
      \end{split}
    \end{equation*}
    By the strong Markov property again, this time at time \( R_1 =H_S\),
    \begin{equation*}
      \begin{split}
        P_x^+\big[ R_1<\infty, (X_{R_1},\dots,X_{D_1}
        &)=w^j \big] \\
        &\!= P_x^+[H_S<\infty, X_{H_S}=w^j(0) ]\,
          P^+_{w^j(0)}\big[ (X_0,\dots,X_{T_U})=w^j \big] \\
        &\!= \frac{P_x^+[H_S<\infty, X_{H_S}=w^j(0) ]}{\tilde{e}_S(w^j(0))}\,
          P^+_{\tilde{e}_S}\big[ (X_0,\dots,X_{T_U})=w^j \big] \\
        &\leq \frac{\delta}{2e^4} \,\Gamma (w^j),
      \end{split}
    \end{equation*}
    where we are using the inequality
    \begin{equation}
      \label{ec:beta/1-beta}
      P_x^+(X_{H_S} = y) \leq \frac{\delta}{2e^4}\, \tilde{e}_S(y),\quad  y\in S, 
    \end{equation}
    to be proved right below (we omit the condition \( \{H_S<\infty\} \) as it is
    implicit). Replacing back in the previous expression finishes the induction. Observe that in
    case \( \tilde{e}_S(w^j(0))=0 \), the inequality is trivially satisfied.

    \par Let \( y\in S \) be arbitrary. In a similar way as we did before, decompose the event
    \( \{X_{H_S} = y\} \) by the last visit to \( V^c \) and then apply the Markov
    property,
    \begin{equation}
      \label{eq:beta}
      \begin{split}
        P_x^+(&X_{H_S} = y) \\
              &= \sum_{z\in \partial^oV}\, \sum_{n\geq 0}\, P_x^+(X_n=z, n < H_S, X_{n+1} \in V, X_{n+2} \in V, \dots, X_{H_S}=y\in V) \\
              &= \sum_{z \in \partial^oV}\, \sum_{n\geq 0}\, P_x^+(X_n=z, n < H_S)\; P_z^+( X_1 \in V, X_2 \in V, \dots, X_{H_S}=y\in V) \\
              &= \sum_{z \in \partial^oV}\, P_x^+(H_z <\infty , H_z < H_S)\; P_z^+(X_{\widetilde{H}_{S\cup V^c}}=y) \\
              &\leq \max_{z \in \partial^oV}\, P_x^+(H_z <\infty)\; \sum_{z\in \partial^oV} P_y^-(X_{\widetilde{H}_{S\cup V^c}}=z) .
      \end{split}
    \end{equation}
    In the last line we used that
    \begin{equation*}
     P_z^+(X_{\widetilde{H}_{S\cup V^c}}=y)\, =  P_y^-(X_{\widetilde{H}_{S\cup V^c}}=z) ,
    \end{equation*}
    for paths traveled in reverse time have the same probability. Using the strong
    Markov property again,
    \begin{equation*}
      \begin{split}
        e_S(y) = P^-_y(\widetilde{H}_S=\infty)
        &= \sum_{z\in \partial ^oV} P_y^-(X_{\widetilde{H}_{S\cup V^c}}=z, H_{S}\circ\theta_{\widetilde{H}_{S\cup V^c}}=\infty) \\
        &= \sum_{z\in \partial ^oV} P_y^-(X_{\widetilde{H}_{S\cup V^c}}=z) P^-_z(H_{S}=\infty) \\
        &\geq \min_{z\in \partial ^oV} P^-_z(H_{S}=\infty) \sum_{z\in \partial ^oV} P_y^-(X_{\widetilde{H}_{S\cup V^c}}=z) .
      \end{split}
    \end{equation*}
    Proceeding as in (\Ref{eq:1-beta}) we obtain
    \( \max_{z\in \partial^o V}P^-_z(H_S<\infty)\leq 1/2 \) or,
    \( \min_{z\in \partial ^oV} P^-_z(H_{S}=\infty)\geq 1/2 \). Now replacing back in
    (\Ref{eq:beta}) we arrive at
    \begin{equation*}
      \begin{split}
        P_x^+(X_{H_S} = y)
        &\leq \max_{z \in \partial^oV}\, P_x^+(H_z <\infty) \, 2\,e_S(y) \\
        &= 2\max_{z \in \partial^oV}\, g(x,z) \, \mathrm{cap}(S) \, \tilde{e}_S(y) \\
        &\leq \frac{\delta}{2e^4}\tilde{e}_S(y).
      \end{split}
    \end{equation*}
    In the last line we used the same reasoning as in the last line of (\Ref{eq:1-beta}) but didn't
    use the bound on \( \delta \). This finishes the proof of the second inequality.
  \end{proof}

  There exists a probability space,
  \( (\overline{\Omega},\, \overline{\!\!\mathcal{A}}, \overline{\mathbb{P}}) \), where we can
  define independent Poisson random variables \( N^1_{-,+} \sim \mathrm{Pois}(\lambda^1) \),
  \( N_-^j \sim \mathrm{Pois}(\lambda^j) \), \( j\geq 2 \), with
  \begin{equation*}
    \lambda^1 = \frac{1}{2}(u_+ -u_-)\,\mathrm{cap}(S), \quad  \quad \lambda^j = \left(\frac{\delta}{2e^4}\right)^{j-1}
      u_-\,\mathrm{cap}(S),
  \end{equation*}
  and \( \mathcal{C}_{S,U^c} \)-valued \( \gamma_1,\gamma_2, \dots \), independent and identically distributed as
  \( \Gamma \) (observe that the Poisson parameters are positive numbers).  Using
  \( \gamma_1, \gamma_2, \dots \), successively define samples of size \( N_-^j \) of \( j \)
  excursions as follows
  \begin{align*}
    & (\gamma_{{}_1},\gamma_{{}_2}), \dots ,(\gamma_{{}_{M_2-1}},\gamma_{{}_{M_2}}) \\
    & (\gamma_{{}_{M_2+1}},\gamma_{{}_{M_2+2}},\gamma_{{}_{M_2+3}}), \dots ,(\gamma_{{}_{M_3-2}},\gamma_{{}_{M_3-1}},\gamma_{{}_{M_3}}) \\
    &\quad \quad \quad \;\;  \dots \\
    & (\gamma_{{}_{M_{j-1}+1}}, \dots ,\gamma_{{}_{M_{j-1}+j}}), \dots ,(\gamma_{{}_{M_j-j+1}}, \dots ,\gamma_{{}_{M_j}}) \\
    &\quad \quad \quad \;\;  \dots\, ,
  \end{align*}
  where \( M_k = \sum_{j=2}^k jN_-^j \). Rows are independent and the elements of row \( j \) are
  independent and identically distributed as \( \Gamma^{\otimes j} \), call them
  \( \gamma _{j,1}, \dots \gamma _{j,N_-^j} \). For \( j\geq 2 \),
  \begin{equation*}
    \sigma _-^j = \sum_{i=1}^{N_-^j} \,\delta _{\gamma_{j,i}}
  \end{equation*}
  are independent Poisson point processes with intensity measure
  \( \lambda^j\,\Gamma ^{\otimes j} \). Let \( \sigma^1 \) be an
  independent Poisson point process on \( \mathcal{C}_{S,U^c} \) with intensity measure
  \begin{equation*}
    \xi^1_{-,+} - \lambda^1\,\Gamma,
  \end{equation*}
  and \( \{\alpha _{i,j}\}_{1\leq i\leq N_-^j, j\geq 2} \) also independent, conditional on
  \( (\gamma_k)_{k\geq 1} \) and \( (N_-^j)_{j\geq 2} \) Bernoulli random variables with parameter
  \begin{equation*}
    \overline{\mathbb{P}}(\alpha_{i,j} = 1) = \frac{\xi_-^j(\gamma_{i,j})}{\lambda^j\,\Gamma ^{\otimes j}(\gamma_{i,j})}\, .
  \end{equation*}
  That this choice of intensity measure and parameters is allowed is the content of Lemma
  \Ref{lm:indep-excursions-ineq}.

  \par We can now define, for \( j\geq 2 \), the thinning of \( \sigma _-^j \),
  \begin{equation*}
    \bar{\zeta} _-^j = \sum_{i=1}^{N_-^j} \,\alpha _{i,j}\,\delta _{\gamma_{i,j}},
  \end{equation*}
  which we see are independent Poisson point processes with intensity measure \( \xi_-^j \), implying
  that \( (\bar{\zeta} _-^j)_{j\geq 2} \) is distributed as \( (\zeta_-^j)_{j\geq 2} \). Then, by
  (\Ref{eq:zeta-star-star-minus}),
  \begin{equation*}
    \sum_{j\geq 2} s_j(\bar{\zeta}^j_-)
  \end{equation*}
  is distributed as \( \zeta^{* *}_- \).  Similarly, as the Poisson point process
  \begin{equation*}
    \Sigma^1_{-,+} = \sum_{k=1}^{N^1_{-,+}} \delta_{\gamma_k}
  \end{equation*}
  has intensity measure \( \lambda^1\,\Gamma \), the intensity measure of
  \( \Sigma^1_{-,+} + \sigma^1 \) is \( \xi^1_{-,+} \). Therefore the coupling we are looking for
  is
  \begin{equation*}
    \bar{\zeta}^*_{-,+} = \Sigma^1_{-,+} + \sigma ^1 \quad  \text{ and } \quad
    \bar{\zeta}^{* *}_- = \sum_{j\geq 2} s_j(\bar{\zeta} ^j_-).
  \end{equation*}
  Now, as \( \bar{\zeta}^j_- \leq \sigma^j_- \), it follows that
  \begin{equation*}
    \bar{\zeta}^{* *}_- \leq \sum_{j\geq 2} s_j(\sigma^j_-) = \sum_{k=1} ^{N_-} \delta_{\gamma _k},
  \end{equation*}
  where \( N_-=\sum_{j\geq 2}jN^j_- \). It is also clear that
  \begin{equation*}
    \sum_{k=1}^{N^1_{-,+}} \delta_{\gamma_k} = \Sigma^1_{-,+} \leq \bar{\zeta}^*_{-,+}. 
  \end{equation*}
  We obtain the following inclusion,
  \begin{equation}
    \label{eq:inclusion-coupling}
    \left\{N_- \leq  N^1_{-,+}\right\} \subset \left\{\sum_{k=1} ^{N_-} \delta_{\gamma _k} \leq  \sum_{k=1}^{N^1_{-,+}} \delta_{\gamma_k} \right\} \subset \left\{ \bar{\zeta}^{**}_- \leq \bar{\zeta}^*_{-,+}\right\}.
  \end{equation}
  To conclude the proof of the theorem we just need to bound
  \begin{equation*}
    \overline{\mathbb{P}}(N_{-,+}^1 < N_-) \leq \overline{\mathbb{P}}\left(N_{-,+}^1 \leq \frac{\lambda^1}{2}\right) + \overline{\mathbb{P}}\left( \frac{\lambda^1}{2} \leq N_-\right).
  \end{equation*}
  The heuristic here is that we can make the means
  \( \lambda ^1 = (1/2)(u_+ - u_-) \,\mathrm{cap}(S) \) and
  \( \sum_{j\geq 2} j\lambda^j = u_- \,\mathrm{cap}(S) \sum_{j\geq 2}j(\delta/2e^4)^{j-1} \)
  of \( N_{-,+}^1 \) and \( N_- \) large by making \( \mathrm{cap}(S) \) large. This would imply that
  \( N_{-,+}^1 \) and \( N_- \) will be close to their means which, although they have the same
  order of magnitude, they are not close to each other.
  
  \par By the Markov inequality
  \begin{equation*}
    \begin{split}
      \overline{\mathbb{P}}\left(N_{-,+}^1 \leq \frac{\lambda^1}{2}\right)
      &=  \overline{\mathbb{P}}\left(e^{-\frac{\lambda^1}{2}} \leq e^{-N_{-,+}^1}\right) \\
      &\leq e^{\frac{\lambda^1}{2}}\, \overline{\mathbb{E}} \left[e^{-N_{-,+}^1}\right] \\
      &= e^{\lambda ^1\left(\frac{1}{e}-\frac{1}{2}\right)} \leq e^{-\lambda^1 /10}.
    \end{split}
  \end{equation*}
  Similarly
  \begin{equation*}
    \begin{split}
      \overline{\mathbb{P}}\left( \frac{\lambda^1}{2} \leq N_- \right)
      &=  \overline{\mathbb{P}}\left( e^{\frac{\lambda^1}{2}} \leq e^{N_-} \right) \\
      &\leq e^{-\frac{\lambda^1}{2}}\, \overline{\mathbb{E}} \left[e^{N_-}\right] \\
      &= e^{-\frac{\lambda^1}{2} + \sum_{j\geq 2} \lambda^j(e^j-1)} \leq  e^{-\lambda^1 /10},
    \end{split}
  \end{equation*}
  where we used that \( \sum_{j\geq 2} \lambda^j(e^j-1) \leq \lambda^1 /e \), as the following simple calculation shows
  \begin{equation*}
    \begin{split}
      \sum_{j\geq 2} \lambda^j(e^j-1)
      &\leq u_-\mathrm{cap}(S) \sum_{j\geq 2} \left(\frac{\delta}{2e^4}\right) ^{j-1}e^j\\
      &= u_-\mathrm{cap}(S) \frac{\delta}{2e^2} \sum_{j\geq 0} \left(\frac{\delta}{2e^3}\right)^j \\
      &\leq \frac{1}{e}u_-\mathrm{cap}(S)\frac{\delta}{e} \\
      &\leq \frac{\lambda^1}{e}.
    \end{split}
  \end{equation*}
  We have proved the inequality
  \begin{equation*}
    \overline{\mathbb{P}}(N_{-,+}^1 < N_-) \leq 2\,e^{- \frac{\lambda^1}{10}},
  \end{equation*}
  and therefore also the theorem.
\end{proof}

\printbibliography

\end{document}